\documentclass[11pt]{article}

\usepackage[T1]{fontenc}
\usepackage[utf8]{inputenc}
\usepackage[british]{babel}
\usepackage[
rm={oldstyle=false,proportional=true},
sf={oldstyle=false,proportional=true},
tt={oldstyle=false,proportional=true}
]{cfr-lm}
\usepackage[babel]{microtype}

\usepackage{pifont}
\newcommand{\cmark}{\ding{51}}
\newcommand{\xmark}{\ding{55}}

\usepackage{amsmath}
\usepackage{amsthm}
\usepackage{amssymb}
\usepackage{mathtools}
\usepackage{mathrsfs}

\usepackage{enumitem}
\usepackage{multicol}

\usepackage[noadjust]{cite}
\usepackage[hidelinks,bookmarksnumbered]{hyperref}
\usepackage[capitalise]{cleveref}

\usepackage{orcidlink}

\theoremstyle{plain}
\newtheorem{theorem}{Theorem}[section]
\newtheorem{lemma}[theorem]{Lemma}
\newtheorem{proposition}[theorem]{Proposition}
\newtheorem{corollary}[theorem]{Corollary}

\newcommand{\thistheoremname}{}
\newtheorem*{genericthm*}{\thistheoremname}
\newenvironment{namedthm*}[1]
{\renewcommand{\thistheoremname}{#1}%
\begin{genericthm*}}
{\end{genericthm*}}

\theoremstyle{definition}
\newtheorem{definition}[theorem]{Definition}

\newtheorem{problem}[theorem]{Open problem}

\DeclareMathOperator{\birth}{\tilde{b}}
\DeclareMathOperator{\outcome}{o}
\DeclareMathOperator{\outcomeL}{o^L}
\DeclareMathOperator{\outcomeR}{o^R}
\DeclareMathOperator{\maug}{\mathcal{M}_{aug}}

\makeatletter
\newcommand{\tomb}{
    \begin{tikzpicture}[scale=0.11]
        \filldraw (0,0) rectangle (1,2);
    \end{tikzpicture}
    \@ifnextchar,{\hspace{0.1em}}{}%
}
\makeatother

\begin{document}
\title{Invertibility in the mis\`ere multiverse}
\author{Alfie
    Davies\,\orcidlink{0000-0002-4215-7343}\\\small{\emph{\href{mailto:research@alfied.xyz}{research@alfied.xyz}}}
\and Vishal
Yadav\,\orcidlink{0009-0005-4583-9429}\\\small{\emph{\href{mailto:vkyadav@mun.ca}{vkyadav@mun.ca}}}}
\date{
    \small{
        Department of Mathematics and Statistics,\\
        Memorial University of Newfoundland,\\
        Canada
    }
}
\maketitle

\begin{abstract}
    Understanding invertibility in restricted mis\`ere play has been
    challenging; in particular, the possibility of non-conjugate inverses posed
    difficulties. Advances have been made in a few specific universes, but a
    general theorem was elusive. We prove that every universe has the conjugate
    property, and also give a characterisation of the invertible elements of
    each universe. We then explore when a universe can have non-trivial
    invertible elements, leaving a slew of open problems to be further
    investigated.
\end{abstract}

\section{Introduction}

The current direction of mis\`ere research is moving fiercely towards studying
restrictions of the full mis\`ere structure $\mathcal{M}$. Larsson, Nowakowski,
and Santos, in their work on absolute combinatorial game theory, introduced the
notion of an (absolute) \emph{universe}\footnote{
    We will follow Siegel \cite[Footnote 1 on p.~2]{siegel:on} in dropping the
    adjective `absolute' from `absolute universe', and note that a universe (in
    the original, weaker sense) is absolute if and only if it is parental (see
    \cite[Definition 23, Theorem 24, and Corollary 27 on
    pp.~21--23]{larsson.nowakowski.ea:absolute}).
} \cite[Definitions 12 and 13 on pp.~17--18]{larsson.nowakowski.ea:absolute},
which is a set of games that satisfies additive, conjugate, hereditary, and
dicotic closure. These closure properties, aside from the dicotic closure, are
naturally occurring in many rulesets that are typically studied, and so these
universes are not entirely arbitrary sets of games to study.

A terrific tool was built in \cite[Theorem 4 on
p.~6]{larsson.nowakowski.ea:absolute} that yields a comparison test when
working modulo a universe, which is analogous to the test one enjoys when
working in normal play---this cements the concept of the universe as warranting
further study. It was then shown in \cite[Theorem 26 on
p.~13]{larsson.nowakowski.ea:infinitely}, again by Larsson, Nowakowski, and
Santos, that there are infinitely many (mis\`ere) universes. In fact, Siegel
strengthened this and showed that there are uncountably many universes lying
between the dicot universe $\mathcal{D}$ and the dead-ending universe
$\mathcal{E}$ \cite[pp.~10--11]{siegel:on}.\footnote{
    There is some subtlety here, however, about when two universes should be
    called `distinct', and this is not yet a well-trodden idea in the
    literature: should they be distinct if they are different sets of games, or
    should they be distinct if they have different equivalence classes, et
    cetera. See \cite[Definition 12 and Observation 13 on
    p.~7]{larsson.nowakowski.ea:infinitely}.
}

It is immediate that every universe is a monoid, and hence it is natural to ask
about invertibility. Indeed, Milley writes in \cite[p.~13]{milley:partizan}
that ``[a] better understanding of mis\`ere invertibility is a significant open
problem in the growing theory of restricted mis\`ere play.'' In the near-decade
since \cite{milley:partizan} was published, there have been numerous
improvements. For instance, it was shown by Larsson, Milley, Nowakowski,
Renault, and Santos in \cite[Theorems 36 and 37 on
pp.~20--22]{larsson.milley.ea:progress} that both $\mathcal{D}$ and
$\mathcal{E}$ have the conjugate property \cite[Definition 35 on
p.~20]{larsson.milley.ea:progress}; that is, if $G\in\mathcal{D}$ and there
exists some $H\in\mathcal{D}$ such that $G+H\equiv_\mathcal{D}0$, then
$H\equiv_\mathcal{D}\overline{G}$, and similarly for $\mathcal{E}$.
Furthermore, characterisations were then found for the $\mathcal{D}$-invertible
elements of $\mathcal{D}$ by Fisher, Nowakowski, and Santos \cite[Theorem 12 on
p.~7]{fisher.nowakowski.ea:invertible}, and also for the
$\mathcal{E}$-invertible elements of $\mathcal{E}$ by Milley and Renault
\cite[Theorems 19 and 22 on pp.~11--12]{milley.renault:invertible}.

In their survey on partizan mis\`ere theory, Milley and Renault posed two open
problems \cite[\S7 on p.~122]{milley.renault:restricted} The first concerned
the possibility of non-conjugate inverses. They asked ``[in] what universes
does $G+H\equiv_\mathcal{U}0$ imply $\overline{G}\equiv_\mathcal{U}H$?'' In
particular, they wonder whether the implication can be proved for parental,
dense universes (they are using ``universe'' in the weaker sense). This is
equivalent to asking if it is true for every (absolute) universe (going back to
our terminology). Such questions have been asked before. Indeed, Milley
conjectured in \cite[Conjecture 2.1.7 on p.~10]{milley:restricted} that being
closed under conjugation would suffice to yield the conjugate property. This
turned out to be false, as discussed again by Milley in
\cite[p.~13]{milley:partizan} where they then asked ``is there some condition
on the universe $\mathcal{U}$ so that $G$ being invertible implies
$G+\overline{G}\equiv_\mathcal{U}0$?'' The answer is yes, and this is what we
now prove a decade or so later!

Notably, our results here were not formed from a steady trickle of small
improvements, but rather were hammered into existence by a recent discovery:
Siegel's simplest forms \cite[\S5]{siegel:on}. These forms allow us to do away
with the complications that can arise from reversing through ends, and indeed
they effectively solve the second open problem posed in \cite[\S7 on
p.~122]{milley.renault:restricted} as it pertains to (absolute) universes. We
will give the necessary background for this theory, as well as other important
ideas, in \cref{sec:prelims}.

In \cref{sec:invertibility}, we make the strides towards the better
understanding of invertibility that we have just discussed: we show that every
universe has the conjugate property (\cref{thm:conjugate-property}). We also
give a characterisation of the invertible elements of each universe
(\cref{thm:invertibility-characterisation}), which we then contrast with the
known characterisations in $\mathcal{D}$ and $\mathcal{E}$.

In \cref{sec:reduced}, we explore when a universe has no (non-trivial)
invertible elements; we call such a universe \emph{reduced}\footnote{
    A monoid whose invertible subgroup is trivial is often called reduced, and
    we use that same language for our universes here (since every universe is a
    monoid).
}. We give a characterisation to test for when a universe is reduced, although
the result is far from practical.

In \cref{sec:weak}, we introduce the notion of a \emph{weak} universe (and also
a weak set of games), which we will see is equivalent to meaning one that
induces the same partial order relation as full mis\`ere. (The actual
definition requires more material than we have in this introduction; see
\cref{def:weak}.) When a universe is weak, it can have no (non-zero) invertible
elements; it is reduced. But it is unclear whether the reverse implication also
holds. We discover elements, such as $\{\cdot\mid2\}$, whose mere presence in a
universe is enough to render it weak. Speaking practically, if a universe
contains $\{\cdot\mid2\}$, then it contains no (non-zero) invertible elements.

Finally, in \cref{sec:final}, we briefly discuss the landscape and set some
future directions.

\section{Preliminaries}
\label{sec:prelims}

The main object of our study here is the universe, and so we had better define
it! We will not, however, recall all of the basic ideas of the field of
Combinatorial Game Theory, and the reader is directed to Siegel's wonderful
book for such information \cite{siegel:combinatorial}.

A set of games $\mathcal{U}$ is called a \emph{universe} if the following
closure properties are satisfied for all $G,H\in\mathcal{U}$:
\begin{itemize}
    \item
        (additive closure) $G+H\in\mathcal{U}$;
    \item
        (conjugate closure) $\overline{G}\in\mathcal{U}$;
    \item
        (hereditary closure) $G'\in\mathcal{U}$ for all options $G'$ of $G$;
    \item
        (dicotic closure) if $\mathscr{G},\mathscr{H}\subseteq\mathcal{U}$ are
        finite and non-empty, then
        \[
            \{\mathscr{G}\mid\mathscr{H}\}\in\mathcal{U}.
        \]
\end{itemize}
Given a universe $\mathcal{U}$, Siegel defines a game $G$ to be Left
$\mathcal{U}$-strong if $\outcome(G+X)\geq\mathscr{N}$ for all Left ends
$X\in\mathcal{U}$ \cite[Definition 2.3 on p.~5]{siegel:on}.

Now, as we alluded to in the introduction, our results make heavy use of
Siegel's simplest forms. These are new ideas in mis\`ere theory, and the reader
is most strongly encouraged to read Siegel's original presentation of the
material \cite[\S5]{siegel:on} that includes excellent motivations and
additional information to what we will provide here; we include only enough to
get by.

Say we have a universe $\mathcal{U}$, and some game $G=\{G^L\mid G^R\}$. A Left
option $G^L$ of $G$ is $\mathcal{U}$-\emph{reversible} if there exists some
$G^{LR}$ such that $G\geq_\mathcal{U}G^{LR}$. In normal play, when such an
option is reversible, the reduction that occurs is replacing the $G^L$ with all
of the Left options of the $G^{LR}$. When $G^{LR}$ is not a Left end, we may do
the same thing in (restricted) mis\`ere. The problem occurs precisely when
$G^L$ is a Left end. This is because, if $G$ has a $\mathcal{U}$-reversible
option $G^L$ that is a Left end, then $G$ is Left $\mathcal{U}$-strong, and
removing that $G^L$ (i.e.\ replacing it with nothing) may not preserve enough
information to keep the form Left $\mathcal{U}$-strong. But, as Siegel shows,
we can get rid of almost all of the extra information encoded in $G^L$.

The key idea is to remove the end-reversible option and replace it with an
abstract symbol ``\tomb'' called a \emph{tombstone}.\footnote{Siegel actually
    writes $\Sigma^L$ and $\Sigma^R$ for a Left and Right tombstone
    respectively, but this has no consequence for the theory, and we have
    chosen to simplify the notation here and emphasise that this is an abstract
    symbol acting as a flag, rather than possibly some game $\Sigma$ with Left
    or Right options.
}
This tombstone just acts as a marker, or a flag, to the fact that the game is
Left / Right $\mathcal{U}$-strong, and this is what the reader should keep in
their head. Siegel calls the tombstones (if they are present) the
\emph{tombstone options} of a game, and the remaining options are the
\emph{ordinary options}. But, as a reminder, the tombstones are \emph{not}
games, and they cannot be moved to. When we write $G^L$ (or $G^R$), it will
never refer to tombstone options, but instead only to ordinary options.

We write $\maug$ for the set of all game forms where each subposition may be
adorned with a Left or Right tombstone, or both; these are Siegel's
\emph{augmented forms} (cf.\ \cite[Definition 5.1 on p.~21]{siegel:on}). Siegel
then illustrates a beautiful theory, extending the definitions of addition and
outcomes and other useful things. We will extend Siegel's definition of being
Left $\mathcal{U}$-strong to work for arbitrary sets of games, rather than just
universes, since we will have use for it later.

\begin{definition}[cf.\ {\cite[p.~22 and Definition 2.3 on p.~5]{siegel:on}}]
    If $\mathcal{A}$ is a set of games, then we say $G\in\maug$ is \emph{Left
    $\mathcal{A}$-strong} if $\outcomeL(G+X)=\mathscr{L}$ for all Left ends
    $X\in\mathcal{A}$. (We define \emph{Right $\mathcal{A}$-strong}
    analogously.)
\end{definition}

Most importantly, Siegel extends the terrific comparison test of \cite[Theorem
4 on p.~6]{larsson.nowakowski.ea:absolute} to work for all pairs of augmented
forms. If an augmented form is a Left end (i.e.\ there exists no $G^L$) or
contains a Left tombstone, then Siegel calls it \emph{Left end-like} (and
similarly for \emph{Right end-like}).

\begin{theorem}[{\cite[Theorem 5.5 on p.~22]{siegel:on}}]
    \label{thm:comparison}
    If $\mathcal{U}$ is a universe and $G,H\in\maug$, then $G\geq_\mathcal{U}H$
    if and only if:
    \begin{enumerate}[itemsep=8pt, label=(\alph*)]
        \item
            for every $G^R$, either there exists some $H^R$ with
            $G^R\geq_\mathcal{U} H^R$, or else there exists some $G^{RL}$ with
            $G^{RL}\geq_\mathcal{U}H$;

        \item
            for every $H^L$, either there exists some $G^L$ with
            $G^L\geq_\mathcal{U}H^L$, or else there exists some $H^{LR}$ with
            $G\geq_\mathcal{U}H^{LR}$;

        \item
            if $H$ is Left end-like, then $G$ is Left $\mathcal{U}$-strong;

        \item
            if $G$ is Right end-like, then $H$ is Right $\mathcal{U}$-strong.
    \end{enumerate}
\end{theorem}
In \cref{thm:comparison}, conditions (a)--(b) are commonly referred to as the
\emph{maintenance property}, and (c)--(d) as the \emph{proviso}.

If we weaken the requirement of \cref{thm:comparison} that $\mathcal{U}$ be a
universe, then we obtain the following simple results. These results are not
significant developments; we are including them here because writing them down
in an organised way will be useful to us later!

\begin{proposition}
    \label{prop:comp-implies}
    If $\mathcal{A}$ is a set of games, and $G,H\in\maug$ with
    $G\geq_\mathcal{A}H$, then:
    \begin{itemize}
        \item
            if $H$ is Left end-like, then $G$ is Left $\mathcal{A}$-strong; and
        \item
            if $G$ is a Right end-like, then $H$ is Right $\mathcal{A}$-strong.
    \end{itemize}
\end{proposition}

\begin{proof}
    Since $G\geq_\mathcal{A}H$, we must have $\outcome(G+X)\geq\outcome(H+X)$
    for all $X\in\mathcal{A}$ by definition. If $H$ is Left end-like, then
    $\outcome(H+X)\geq\mathscr{N}$, and hence $G$ must be Left
    $\mathcal{A}$-strong. The rest follows by symmetry.
\end{proof}

In the following proposition, we can follow (almost verbatim) one of the
directions of Siegel's proof \cite[Proof of Theorem 5.5 on
pp.~24--25]{siegel:on}.

\begin{proposition}
    \label{prop:maint+strong-implies}
    If $\mathcal{A}$ is a hereditary set of games, and $G,H\in\maug$ satisfy
    each of the following properties, then $G\geq_\mathcal{A}H$:
    \begin{itemize}
        \item
            for all $G^R$, either there exists some $H^R$ with
            $G^R\geq_\mathcal{A}H^R$, or else there exists some $G^{RL}$ with
            $G^{RL}\geq_\mathcal{A}H$;
        \item
            for all $H^L$, either there exists some $G^L$ with
            $G^L\geq_\mathcal{A}H^L$, or else there exists some $H^{LR}$ with
            $G\geq_\mathcal{A}H^{LR}$;
        \item
            if $G$ is Right end-like, then $H$ is Right $\mathcal{A}$-strong;
            and
        \item
            if $H$ is Left end-like, then $G$ is Left $\mathcal{A}$-strong.
    \end{itemize}
\end{proposition}

\begin{proof}
    We need to show that $\outcome(G+X)\geq\outcome(H+X)$ for all
    $X\in\mathcal{A}$. It suffices to show that
    $\outcomeL(G+X)\geq\outcomeL(H+X)$ for all $X\in\mathcal{A}$, since an
    identical argument will prove $\outcomeR(G+X)\geq\outcomeR(H+X)$. We
    proceed by induction on the formal birthday of $G+H+X$.

    We need only consider when $\outcomeL(H+X)=\mathscr{L}$. We must be in one
    of three cases:
    \begin{enumerate}
        \item
            $\outcomeR(H+X^L)=\mathscr{L}$ for some $X^L$.

            By induction, since $\mathcal{A}$ is hereditary, we have that
            $\outcomeR(G+X^L)=\mathscr{L}$, and hence
            $\outcomeL(G+X)=\mathscr{L}$.
        \item
            $\outcomeR(H^L+X)=\mathscr{L}$ for some $H^L$.

            By hypothesis, either there exists some $G^L$ with
            $G^L\geq_\mathcal{A}H^L$, or else there exists some $H^{LR}$ with
            $G\geq_\mathcal{A}H^{LR}$. In the first case, we have immediately
            that $\outcomeR(G^L+X)\geq\outcomeR(H^L+X)$, and hence
            $\outcomeL(G+X)=\mathscr{L}$. In the latter, we
            observe that $\outcomeL(H^{LR}+X)=\mathscr{L}$ (since
            $\outcomeR(H^L+X)=\mathscr{L}$), and so
            $\outcomeL(G+X)\geq\outcome(H^{LR}+X)=\mathscr{L}$.
        \item
            $H+X$ is Left end-like.

            Since $X\in\mathcal{A}$, it must follow that $X$ is a Left end and
            $H$ is Left end-like. By hypothesis, $G$ is Left
            $\mathcal{A}$-strong, and hence $\outcomeL(G+X)=\mathscr{L}$.
    \end{enumerate}
\end{proof}

The other major piece of the theory we will require is the notion of the
$\mathcal{U}$-simplest form (see \cite[Definition 5.19 on p.~27]{siegel:on}).
When we have a universe $\mathcal{U}$, and some augmented form $G$, we can
repeatedly perform the following reductions on all subpositions of $G$:
\begin{enumerate}
    \item
        bypass $\mathcal{U}$-reversible options, replacing end-reversible ones
        with a tombstone;
    \item
        remove $\mathcal{U}$-dominated options; and
    \item
        remove unnecessary tombstones.
\end{enumerate}
Siegel proves that the form we arrive at (i.e.\ when we can no longer perform
any more reductions), which is called the \emph{$\mathcal{U}$-simplest form},
is unique. That is, if $G\equiv_\mathcal{U}H$, then the $\mathcal{U}$-simplest
form of $G$ is isomorphic to the $\mathcal{U}$-simplest form of $H$.

So, when the reader sees a game in $\mathcal{U}$-simplest form in our
arguments, they should have in their mind that there are no
$\mathcal{U}$-dominated options, no $\mathcal{U}$-reversible options, and no
unnecessary tombstones in any of its subpositions.

For each universe $\mathcal{U}$, Siegel then defines a special set of augmented
forms called $\hat{\mathcal{U}}$; simply put, this set contains every augmented
form whose tombstones can be replaced with end-reversible options in
$\mathcal{U}$ such that we can recover a game form in $\mathcal{U}$ (Siegel
says that such a game has a \emph{$\mathcal{U}$-expansion}). So, each game in
$\hat{\mathcal{U}}$ is equal (modulo $\mathcal{U}$) to some game in
$\mathcal{U}$. We will not need the formal definitions and theory here.
Instead, we need just two results.

\begin{lemma}[{\cite[Lemma 5.21 on p.~29]{siegel:on}}]
    If $\mathcal{U}$ is a universe and $G\in\mathcal{U}$, then the
    $\mathcal{U}$-simplest form of $G$ is in $\hat{\mathcal{U}}$.
\end{lemma}

\begin{theorem}[{\cite[Theorem 5.23 on p.~30]{siegel:on}}]
    \label{thm:monoid}
    If $\mathcal{U}$ is a universe, and $G,H,J\in\hat{\mathcal{U}}$ satisfy
    $G\equiv_\mathcal{U}H$, then $G+J\equiv_\mathcal{U}H+J$.
\end{theorem}

So, Siegel shows that $(\hat{\mathcal{U}},\equiv_\mathcal{U})$ is a monoid (in
particular, note that $\hat{\mathcal{U}}$ is closed under addition). We give a
straightforward extension of this to show that
$(\hat{\mathcal{U}},\geq_\mathcal{U})$ is a pomonoid, which is more apt for our
purposes.

\begin{theorem}
    \label{thm:pomonoid}
    If $\mathcal{U}$ is a universe, then $(\hat{\mathcal{U}},\geq_\mathcal{U})$
    is a pomonoid.
\end{theorem}

\begin{proof}
    We need to prove that, for all $G,H,J\in\hat{\mathcal{U}}$ with
    $G\geq_\mathcal{U}H$, it follows that $G+J\geq_\mathcal{U}H+J$.

    Since $G,H,J\in\hat{\mathcal{U}}$, we know that we can find
    $G',H',J'\in\mathcal{U}$ such that
    \begin{align*}
        G'&\equiv_\mathcal{U}G,\\
        H'&\equiv_\mathcal{U}H,\text{ and}\\
        J'&\equiv_\mathcal{U}J.
    \end{align*}
    Observe then that
    $G'\equiv_\mathcal{U}G\geq_\mathcal{U}H\equiv_\mathcal{U}H'$. Since
    $(\mathcal{U},\geq_\mathcal{U})$ is a pomonoid, it then follows from
    \cref{thm:monoid} that
    \[
        G+J\equiv_\mathcal{U}G'+J
        \equiv_\mathcal{U}G'+J'\geq_\mathcal{U}H'+J'\equiv_\mathcal{U}H'+J
        \equiv_\mathcal{U}H+J.
    \]
\end{proof}

\section{Invertibility}
\label{sec:invertibility}

Armed with the powerful tools referenced in our preliminaries, we now seek to
further understand how invertibility works within a universe. Of course, we
must define what we mean by \emph{invertibility}! We don't restrict ourselves
to universes, since we will have something to say later about more general
monoids.

\begin{definition}
    \label{def:invertible}
    If $\mathcal{A}$ is a set of games, $G\in\maug$, and there exists some
    $H\in\mathcal{A}$ such that $G+H\equiv_\mathcal{A}0$, then we say $G$ is
    \emph{$\mathcal{A}$-invertible}; we call $H$ an
    \emph{$\mathcal{A}$-inverse} of $G$.
\end{definition}

For a general set of games $\mathcal{A}$, there doesn't necessarily exist a
\emph{unique} $\mathcal{A}$-inverse for an $\mathcal{A}$-invertible element,
which is why we called $H$ \emph{an} $\mathcal{A}$-inverse of $G$ in
\cref{def:invertible}, as opposed to \emph{the} $\mathcal{A}$-inverse of $G$.
But when $\mathcal{A}$ is a monoid, then the $\mathcal{A}$-inverse of a game
\emph{must} be unique (up to equivalence modulo $\mathcal{A}$, of course).

It is conceivable that one would like to consider when $G+H\equiv_\mathcal{A}0$
for completely arbitrary forms; i.e.\ without restricting $H$ to be an element
of $\mathcal{A}$ like in \cref{def:invertible}. Doing so here would add
difficulties to our exposition, and our results do not currently apply to such
a general case. (But it \emph{is} simpler for us to allow $G$ to be outside of
$\mathcal{A}$, which is why we do so.) In the future, such things may be
further explored.

Readers familiar with normal play will recall that the inverse of a game is
always its conjugate. But, in mis\`ere, we have already discussed that this is
not so obvious a claim.\footnote{It is so non-obvious that it is \emph{false}
    in some cases! This was mentioned in
    \cite[p.~13]{milley.renault:restricted}, where the offending example comes
    from an element of period 6 in a monoid in \cite[A.6 on
    p.~617]{plambeck.siegel:misere}.
}
As such, we give a separate definition for this case. As we mentioned in the
introduction, the idea of conjugate invertibility is not new, but we state this
definition for clarity.

\begin{definition}
    \label{def:conj-invertible}
    If $\mathcal{A}$ is a set of games, and $G\in\maug$, then we say $G$ is
    \emph{conjugate $\mathcal{A}$-invertible} if $\overline{G}$ is an
    $\mathcal{A}$-inverse of $G$.
\end{definition}

One could rephrase this as: we say $G$ is conjugate $\mathcal{A}$-invertible if
$\overline{G}\in\mathcal{A}$ and $G+\overline{G}\equiv_\mathcal{A}0$.

We will write $\mathcal{A}^\times$ and $\mathcal{A}^{\overline{\times}}$ for
the set of $\mathcal{A}$-invertible and conjugate $\mathcal{A}$-invertible
elements of $\mathcal{A}$ respectively; note that we are \emph{not} considering
those forms outside of $\mathcal{A}$ here. Of course, every conjugate
$\mathcal{A}$-invertible game is $\mathcal{A}$-invertible, and so
$\mathcal{A}^{\overline{\times}}\subseteq\mathcal{A}^\times$. We can also make
a trivial remark about 0 being conjugate $\mathcal{A}$-invertible.

\begin{proposition}
    \label{prop:0-invertible}
    If $\mathcal{A}$ is a set of games containing 0, then $0$ is conjugate
    $\mathcal{A}$-invertible.
\end{proposition}

\begin{proof}
    Observe that $\overline{0}\cong0\in\mathcal{A}$, and so also
    $0+\overline{0}\cong0\equiv_\mathcal{A}0$.
\end{proof}

Given two groups $\mathcal{G}$ and $\mathcal{H}$, recall the notation
`$\mathcal{G}\leq\mathcal{H}$' used to mean that $\mathcal{G}$ is a subgroup of
$\mathcal{H}$. Also recall that the set of invertible elements of a monoid
always forms a group. Now, if $\mathcal{A}$ is a monoid of games, is it true
that $\mathcal{A}^{\overline{\times}}\leq\mathcal{A}^\times$? In particular, if
$G,H\in\mathcal{A}^{\overline{\times}}$, is it true that
$G+H\in\mathcal{A}^{\overline{\times}}$? Indeed, the answer is yes.

\begin{proposition}
    \label{prop:conj-subgroup}
    If $\mathcal{A}$ is a monoid of games, then
    $\mathcal{A}^{\overline{\times}}\leq\mathcal{A}^\times$.
\end{proposition}

\begin{proof}
    Let $G,H\in\mathcal{A}^{\overline{\times}}$. Since
    $\mathcal{A}^{\overline{\times}}$ is clearly closed under conjugation, we
    have $\overline{G},\overline{H}\in\mathcal{A}^{\overline{\times}}$, and in
    particular that
    $\overline{G+H}\cong\overline{G}+\overline{H}\in\mathcal{A}$. Now observe
    \begin{align*}
        G+H+\overline{G+H}&\cong G+H+\overline{G}+\overline{H}\\
                          &\cong G+\overline{G}+H+\overline{H}\\
                          &\equiv_\mathcal{A}0.
    \end{align*}
\end{proof}

Given \cref{prop:conj-subgroup}, it is natural to wonder whether
$\mathcal{A}^\times=\mathcal{A}^{\overline{\times}}$. In general, such a
statement is not true. We give a definition here for when the
$\mathcal{A}$-invertible elements (of a set of games $\mathcal{A}$) do coincide
precisely with the conjugate $\mathcal{A}$-invertible elements; this is
essentially the \emph{conjugate property} that other authors mention, but here
we have a different notation and a slightly more general context.

\begin{definition}
    If $\mathcal{A}$ is a set of games such that
    $\mathcal{A}^{\overline{\times}}=\mathcal{A}^\times$, then we say
    $\mathcal{A}$ has the \emph{conjugate property}.
\end{definition}

We will begin to show now that every universe has the conjugate property. In
order to do so, we will first prove a technical lemma that will also be helpful
to us soon after in our characterisation of the invertible elements of each
universe. The proof of this lemma uses a technique from the beginning of a
proof of Ettinger's \cite[Proof of Theorem 14 on pp.~48--51]{ettinger:topics},
and this technique has been mentioned and used several times in these kinds of
arguments (for example, see \cite[Proofs of Theorems 36 and 37 on
pp.~20--23]{larsson.milley.ea:progress}).

\begin{lemma}
    \label{lem:tech}
    If $\mathcal{U}$ is a universe, and $G,H\in\hat{\mathcal{U}}$ are in
    $\mathcal{U}$-simplest form with $G+H\equiv_\mathcal{U}0$, then: for all
    $G^L$, there exists $H^R$ such that $G^L+H^R\equiv_\mathcal{U}0$.
\end{lemma}

\begin{proof}
    If $G$ is a Left end, then the conclusion is vacuously true. So, assume
    that $G$ is not a Left end.

    Since $G+H\equiv_\mathcal{U}0$, we know in particular that
    $G+H\leq_\mathcal{U}0$. Pick an arbitrary Left option $G^{L_1}$. By
    \cref{thm:comparison}, it follows that there must exist some Right option
    $(G^{L_1}+H)^{R_1}$ of $G^{L_1}+H$ such that
    $(G^{L_1}+H)^{R_1}\leq_\mathcal{U}0$. Such a Right option must be of one of
    the following forms:
    \begin{multicols}{2}
        \begin{enumerate}
            \item
                $G^{L_1R_1}+H$; or
            \item
                $G^{L_1}+H^{R_1}$.
        \end{enumerate}
    \end{multicols}
    In the first case, where we have $G^{L_1R_1}+H\leq_\mathcal{U}0$, we may
    add $G$ to both sides to obtain $G^{L_1R_1}\leq_\mathcal{U}G$. (Note that
    we are allowed to add $G$ to both sides since
    $(\hat{\mathcal{U}},\geq_\mathcal{U})$ is a pomonoid by
    \cref{thm:pomonoid}, and $G^{L_1R_1},G,H\in\hat{\mathcal{U}}$.) This tells
    us that $G^{L_1}$ is a reversible Left option of $G$. But this contradicts
    $G$ being in simplest form. Thus, we must be in the second case: the Right
    option must be of the form $G^{L_1}+H^{R_1}$.

    So, recall that we have $G^{L_1}+H^{R_1}\leq_\mathcal{U}0$. If we have
    equality, then we have the result. As such, we will suppose, for a
    contradiction, that we have the strict inequality
    $G^{L_1}+H^{R_1}<_\mathcal{U}0$. Now, given $H^{R_1}$, we can apply a
    symmetric argument to the one before to obtain some $G^{L_2}$ such that
    $G^{L_2}+H^{R_1}\geq_\mathcal{U}0$. We can clearly continue in this way so
    that, given $G^{L_i}$, we may find some $H^{R_i}$ such that
    \begin{equation}
        \label{eq:a}
        G^{L_i}+H^{R_i}\leq_\mathcal{U}0;
    \end{equation}
    and given $H^{R_i}$, we may find some $G^{L_{i+1}}$ such that
    \begin{equation}
        \label{eq:b}
        G^{L_{i+1}}+H^{R_i}\geq_\mathcal{U}0.
    \end{equation}

    We will show by induction that each inequality of type (\ref{eq:b}) is
    strict. If we have equality for $i=1$, then adding $G^{L_1}$ to both sides
    yields $G^{L_2}+G^{L_1}+H^{R_1}\equiv_\mathcal{U}G^{L_1}$. But
    $G^{L_1}+H^{R_1}<_\mathcal{U}0$ by supposition, and so we obtain
    $G^{L_1}\leq_\mathcal{U}G^{L_2}$. Since $G$ is in $\mathcal{U}$-simplest
    form, this implies $G^{L_1}\cong G^{L_2}$. But then we have
    \[
        0\leq_\mathcal{U}G^{L_2}+H^{R_1} \cong G^{L_1}+H^{R_1} <_\mathcal{U}0,
    \]
    which is a contradiction. Hence, $G^{L_2}+H^{R_1}>_\mathcal{U}0$.

    Now suppose $G^{L_{i+1}}+H^{R_i}>_\mathcal{U}0$ for all $i\leq k$. Suppose
    further, for a contradiction, that
    $G^{L_{k+2}}+H^{R_{k+1}}\equiv_\mathcal{U}0$. We may add $G^{L_{k+2}}$ to
    both sides of $G^{L_{k+1}}+H^{R_{k+1}}\leq_\mathcal{U}0$ to obtain
    $G^{L_{k+1}}\leq_\mathcal{U}G^{L_{k+2}}$. Since $G$ is in
    $\mathcal{U}$-simplest form, we must have
    $G^{L_{k+1}}\cong_\mathcal{U}G^{L_{k+2}}$.

    We have $G^{L_{k+1}}+H^{R_k}\geq_\mathcal{U}0$ from (\ref{eq:b}). Adding
    $H^{R_{k+1}}$ to both sides yields $G^{L_{k+1}}+H^{R_{k+1}}+H^{R_k}\geq
    H^{R_{k+1}}$. Since $G^{L_{k+2}}+H^{R_{k+1}}\equiv_\mathcal{U}0$ by
    supposition, and $G^{L_{k+1}}\cong G^{L_{k+2}}$ from earlier, we obtain
    $H^{R_k}\geq_\mathcal{U}H^{R_{k+1}}$. Because $H$ is in
    $\mathcal{U}$-simplest form, we must have $H^{R_k}\cong H^{R_{k+1}}$. But
    then observe
    \[
        0\geq_\mathcal{U}G^{L_{k+1}}+H^{R_{k+1}} \cong G^{L_{k+1}}+H^{R_k}
        >_\mathcal{U}0,
    \]
    which is a contradiction.

    Recall that $G$ and $H$ are short games---in particular, they have only a
    finite number of options. Thus, there must exist integers $a<b$ such that
    $G^{L_a}\cong G^{L_b}$. Summing the inequalities of type (\ref{eq:a}) over
    this range, we obtain
    \[
        S_1\coloneq\sum_{i=a}^{b-1}\left(G^{L_i}+H^{R_i}\right)\leq_\mathcal{U}0.
    \]
    Similarly, summing the inequalities of type (\ref{eq:b}), we obtain
    \[
        S_2\coloneq\sum_{i=a}^{b-1}\left(G^{L_{i+1}}+H^{R_i}\right)>_\mathcal{U}0.
    \]
    But $G^{L_a}\cong G^{L_b}$, and so $S_1\cong S_2$. But this implies
    $0<_\mathcal{U}S_1\leq_\mathcal{U}0$, which is a contradiction.
\end{proof}

Recall again that $(\hat{\mathcal{U}},\geq_\mathcal{U})$ is a pomonoid
(\cref{thm:pomonoid}) and, importantly, that every game in $\hat{\mathcal{U}}$
is equal to some game in $\mathcal{U}$. Thus, we may pass freely between the
two; if we have some $G\in\hat{\mathcal{U}}$ and find another
$H\in\hat{\mathcal{U}}$ such that $G+H\equiv_\mathcal{U}0$, then $G$ is
$\mathcal{U}$-invertible, even though $H$ might not be in $\mathcal{U}$.

\begin{theorem}
    \label{thm:conjugate-property}
    Every universe $\mathcal{U}$ has the conjugate property.
\end{theorem}

\begin{proof}
    We know that 0 is conjugate $\mathcal{U}$-invertible by
    \cref{prop:0-invertible}, so let $G\in\hat{\mathcal{U}}$ be a non-zero,
    $\mathcal{U}$-invertible game; write $H$ for the $\mathcal{U}$-inverse of
    $G$ (and note that it must also be non-zero). We may assume that $G$ and
    $H$ are in $\mathcal{U}$-simplest form. We want to show that
    $G+\overline{G}\equiv_\mathcal{U}0$. Since $G+\overline{G}$ is a symmetric
    form, it suffices to show that either $G+\overline{G}\geq_\mathcal{U}0$ or
    $G+\overline{G}\leq_\mathcal{U}0$. It also suffices to show that
    $\overline{G}\equiv_\mathcal{U}H$. (We will make use of each sufficient
    condition.)

    We show first that $H$ and $\overline{G}$ have identical ordinary options;
    we induct on $\birth(G+H)$. By \cref{lem:tech}, for every Left option
    $G^L$, there exists a Right option $H^R$ such that
    $G^L+H^R\equiv_\mathcal{U}0$. By induction, for every $G^L$, we have
    $G^L+\overline{G^L}\equiv_\mathcal{U}0$, and hence $H^R\cong\overline{G^L}$
    (since $G$ and $H$ are both in $\mathcal{U}$-simplest form). By a symmetric
    argument: for every Right option $G^R$, there exists a Left option $H^L$
    such that $G^R+H^L\equiv_\mathcal{U}0$. By induction, we have
    $G^R+\overline{G^R}\equiv_\mathcal{U}0$, and hence
    $H^L\cong\overline{G^R}$. We have now shown that the ordinary options of
    $\overline{G}$ form a subset of the ordinary options of $H$. By an
    identical argument, it follows that the ordinary options of $H$ form a
    subset of the ordinary options of $\overline{G}$.

    Note that, if $\overline{G}$ is comparable with $H$ (modulo $\mathcal{U}$),
    then we are done: say, without loss of generality, that
    $\overline{G}\geq_\mathcal{U}H$, then we obtain
    $G+\overline{G}\geq_\mathcal{U}0$ by adding $G$ to both sides.

    Now, since $\overline{G}$ and $H$ have the same ordinary options, it is
    clear that the only cases where it is not immediately obvious that
    $\overline{G}$ and $H$ are comparable are where $G$ has both tombstones and
    $H$ has none, and the symmetric case where $G$ has no tombstones and $H$
    has both (this follows easily from the observation that removing a Left
    tombstone, or adding a Right tombstone, results in a form at least as good
    for Right, et similia). By symmetry, we need only consider the former case.

    We will show that $G+\overline{G}\geq_\mathcal{U}0$ via
    \cref{thm:comparison}. The argument will be clearer if we write down the
    options of $G+\overline{G}$ explicitly:
    \begin{align*}
        G+\overline{G}&\cong\{\tomb,G^L+\overline{G},G+\overline{G}^L\mid
        G^R+\overline{G},G+\overline{G}^R,\tomb\}\\
                      &\cong\{\tomb,G^L+\overline{G},G+\overline{G^R}\mid
                      G^R+\overline{G},G+\overline{G^L},\tomb\}\\
                      &\cong\{\tomb,G^L+\overline{G},G+H^L\mid
                      G^R+\overline{G},G+H^R,\tomb\}.
    \end{align*}

    First, we show that for every $(G+\overline{G})^L$ there exists some
    $(G+\overline{G})^{LR}\geq_\mathcal{U}0$. Since $G+\overline{G}$ is a
    symmetric form, we may assume that an arbitrary Right option is of the form
    $G^R+\overline{G}$. Since $G+H\equiv_\mathcal{U}0$ by supposition, we know
    by \cref{lem:tech} that there exists some $H^L$ with
    $G^R+H^L\equiv_\mathcal{U}0$. Since $\overline{G}$ has a Left option to
    $H^L$, we are done. Furthermore, since 0 has no Left options, the rest of
    the maintenance property is vacuously satisfied.

    It remains to show the proviso. We know that 0 is Right
    $\mathcal{U}$-strong, which completes one part. For the other, we remark
    that $G+\overline{G}$ must be Left $\mathcal{U}$-strong since it has a Left
    tombstone. Thus, we have the result.
\end{proof}

Unsurprisingly, a mis\`ere monoid can indeed have the conjugate property
without being a universe. The impartial mis\`ere monoid is one such example
\cite[Corollary 38 on p.~23]{larsson.milley.ea:progress}. For another, consider
the monoid of (game) integers $\mathcal{S}=\{n:n\in\mathbb{Z}\}$.\footnote{
    The reader must excuse our abuse of terminology here, but this is just a
    humourous incident.
}
Amusingly, $(\mathcal{S},\equiv_\mathcal{S})$ is isomorphic (under the natural
mapping) to the group of integers $\mathbb{Z}$ (under addition). Thus,
$\mathcal{S}$ has the conjugate property, but is clearly not a universe---it
does not satisfy the dicotic closure property (and this is the only property of
a universe that it does not satisfy, just like the impartial monoid).

With our conjugate property in hand, we can already give a characterisation of
$\mathcal{U}^\times$ for each universe $\mathcal{U}$.

\begin{theorem}
    \label{thm:invertibility-characterisation}
    If $\mathcal{U}$ is a universe and $G\in\hat{\mathcal{U}}$, then: $G$ is
    $\mathcal{U}$-invertible if and only if $G+\overline{G}$ is Left
    $\mathcal{U}$-strong and every option of the $\mathcal{U}$-simplest form of
    $G$ is $\mathcal{U}$-invertible.\footnote{
        For those interested, we used the working term \emph{strong mirror} to
        describe a game $G$ where $G'+\overline{G'}$ was Left
        $\mathcal{U}$-strong for every subposition $G'$ of $G$. In this way, a
        game $G\in\hat{\mathcal{U}}$ is $\mathcal{U}$-invertible if and only if
        its $\mathcal{U}$-simplest form is a strong mirror.
    }
\end{theorem}

\begin{proof}
    First, suppose that $G+\overline{G}$ is Left $\mathcal{U}$-strong and every
    option of the $\mathcal{U}$-simplest form of $G$ is
    $\mathcal{U}$-invertible. We may assume that $G$ is in simplest form. We
    want to show that $G+\overline{G}\equiv_\mathcal{U}0$. Since
    $G+\overline{G}$ is a symmetric form, we need only show that
    $G+\overline{G}\geq_\mathcal{U}0$. We know that $G+\overline{G}$ is Left
    $\mathcal{U}$-strong, hence also Right $\mathcal{U}$-strong, and so it
    suffices to show, without loss of generality, that, for every Right option
    $G^R$ of $G$, there exists a Left option
    $(G^R+\overline{G})^L\geq_\mathcal{U}0$. But $G^R$ is an option of $G$, and
    hence $\mathcal{U}$-invertible by supposition, and hence also conjugate
    $\mathcal{U}$-invertible by \cref{thm:conjugate-property}. Thus,
    $G^R+\overline{G^R}\equiv_\mathcal{U}0$, and we know that $\overline{G^R}$
    is a Left option of $\overline{G}$, which now yields that $G$ is
    $\mathcal{U}$-invertible.

    Now suppose instead that $G\in\hat{\mathcal{U}}$ is
    $\mathcal{U}$-invertible. We know that $G$ is then also conjugate
    $\mathcal{U}$-invertible by \cref{thm:conjugate-property}. We may assume
    that $G$ is in simplest form.

    Since $G+\overline{G}\equiv_\mathcal{U}0$, we know that $G+\overline{G}$ is
    Left $\mathcal{U}$-strong. Let $G^L$ be a Left option of $G$. By symmetry,
    it remains only to show that $G^L$ is $\mathcal{U}$-invertible: by
    \cref{lem:tech}, there exists a Right option $\overline{G}^R$ of
    $\overline{G}$ such that $G^L+\overline{G}^R\equiv_\mathcal{U}0$, yielding
    the result.
\end{proof}

It is most interesting to compare \cref{thm:invertibility-characterisation}
with the characterisation of invertibility in the dicot universe. We restate
the theorem of \cite{fisher.nowakowski.ea:invertible} (with minor modifications
to ease comparison).

\begin{theorem}[{cf.\ \cite[Theorem 12 on
    p.~7]{fisher.nowakowski.ea:invertible}}]
    If $G\in\mathcal{D}$, then $G$ is $\mathcal{D}$-invertible if and only if
    $\outcome(G'+\overline{G'})=\mathscr{N}$ for every subposition $G'$ of the
    $\mathcal{D}$-simplest form of $G$.
\end{theorem}

Since 0 is the only Left end that is also a dicot, a game $G$ is Left
$\mathcal{D}$-strong if and only if $\outcome(G)=\mathscr{N}$. As such, it is
not hard to see that our theorem (\cref{thm:invertibility-characterisation}) is
stated in essentially the same way, and that it is a straightforward
generalisation. The reader is likely thinking ``of course this was the case,
what other form could it have taken?'' And therein lies some intrigue; let us
restate the theorem of \cite{milley.renault:invertible} for invertibility in
the dead-ending universe (with minor modifications again).

\begin{theorem}[{cf. \cite[Theorems 19 and 22 and
    pp.~11--12]{milley.renault:invertible}}]
    If $G\in\mathcal{E}$, then $G$ is $\mathcal{E}$-invertible if and only if
    no subposition of its $\mathcal{E}$-simplest form has outcome
    $\mathscr{P}$.
\end{theorem}

This is peculiar. The similarity between our
\cref{thm:invertibility-characterisation} and this result is not as obvious it
was for the dicot characterisation. They must yield the same invertible forms,
of course, but the characterisation in terms of the simplest form being free of
subpositions with outcome $\mathscr{P}$ is fascinating; to eschew references to
being Left $\mathcal{E}$-strong, and instead make reference only to the
outcomes of the subpositions, results in what we consider to be a cleaner
statement than ours.\footnote{Theorem aesthetics are subjective, of course, but
    we find it undeniable that there is something special about the
    invertbility characterisation of \cite{milley.renault:invertible}.
}

If it were the case that a game in $\hat{\mathcal{E}}$ doesn't have outcome
$\mathscr{P}$ if and only if $G+\overline{G}$ is Left $\mathcal{E}$-strong,
then the two results would indeed by trivial translations of each other. But
this is not the case: for example, 1 does not have outcome $\mathscr{P}$, but
$1+\overline{1}$ is certainly Left $\mathcal{E}$-strong. What \emph{is} true,
however, is that $G+\overline{G}$ being Left $\mathcal{E}$-strong implies $G$
cannot have outcome $\mathscr{P}$. This follows from a result in
\cite{davies.mckay.ea:pocancellation}, but is not too hard to see: the proof is
essentially a local response strategy, where Right can leverage a waiting game
$W_n$ of sufficiently large formal birthday to always follow the local
$\mathscr{P}$-strategy and win $G+\overline{G}+W_n$ going second if $G$ has
outcome $\mathscr{P}$ (thus, $G+\overline{G}$ could not be Left
$\mathcal{E}$-strong).

We leave it as an open problem to investigate when alternative
characterisations of invertibility like this exist.
\begin{problem}
    \label{prob:inverses}
    What alternative characterisations exist for the invertible elements of
    mis\`ere universes? What about for the dicot universe, specifically?
\end{problem}

\section{Reduced universes}
\label{sec:reduced}

Even with the results of \cref{sec:invertibility}, the investigations into
mis\`ere invertibility are not over! There are still many questions to answer.
In this section, we will explore when a universe has no invertible elements at
all.

Recall that a monoid is called reduced if the identity is the only invertible
element. Since a universe of games is a monoid (where 0 is the identity), we
will use the same terminology and call a universe \emph{reduced} if it is
reduced as a monoid. Equivalently, a universe is called reduced if
$\mathcal{U}^\times$ is the trivial group.

The full mis\`ere universe $\mathcal{M}$ is an example of a reduced universe,
thanks to the well-known theorem of Mesdal and Ottaway \cite[Theorem 7 on
p.~5]{mesdal.ottaway:simplification}. There exist many more (uncountably many,
in fact), but we will not prove such a statement until the next section where
we prove something stronger still (see the discussion after
\cref{cor:universe-weak}).

Our first result here concerns the possible formal birthdays of
$\mathcal{U}$-invertible games (for a universe $\mathcal{U}$). It is unclear
whether a universe that is not reduced has invertible elements of every formal
birthday, but we can prove a weaker conclusion.

\begin{proposition}
    \label{cor:invertible-birthdays}
    If $\mathcal{U}$ is a universe that is not reduced, then there exists a
    $\mathcal{U}$-invertible game in $\mathcal{\hat{U}}$ of formal birthday $n$
    for all $n\in\mathbb{N}$.
\end{proposition}

\begin{proof}
    By hypothesis, there exists some non-zero $G\in\mathcal{U}$ such that $G$
    is $\mathcal{U}$-invertible. Since it is non-zero, its
    $\mathcal{U}$-simplest form must have formal birthday at least 1, and this
    form must be $\mathcal{U}$-invertible (because $G$ is). Thus, by the
    characterisation of $\mathcal{U}$-invertibility
    (\cref{thm:invertibility-characterisation}), it follows that there must
    exist a $\mathcal{U}$-invertible subposition $G'\in\hat{\mathcal{U}}$ of
    the $\mathcal{U}$-simplest form of $G$ of formal birthday 1. We then
    observe that $n\cdot G\in\hat{\mathcal{U}}$ is a $\mathcal{U}$-invertible
    game of birthday $n$, yielding the result.
\end{proof}

For a universe $\mathcal{U}$ that is not reduced, we are \emph{not} claiming
that $\mathcal{U}$ contains a $\mathcal{U}$-invertible element of formal
birthday 1; we have proved that $\mathcal{U}$ contains a
$\mathcal{U}$-invertible element whose \emph{$\mathcal{U}$-simplest form} has
formal birthday 1. But the present authors know of no counter-example to the
stronger assertion, which we pose as an open problem.

\begin{problem}
    \label{prob:rank1}
    Does there exist a universe $\mathcal{U}$ that is not reduced, such that
    there exists no $\mathcal{U}$-invertible $G\in\mathcal{U}$ of formal
    birthday 1? Of formal birthday less than or equal to $n$?
\end{problem}

As a result of \cref{cor:invertible-birthdays}, to characterise those universes
that are reduced monoids, we can look at all of the augmented forms born by day
1 and try to characterise which universes they are invertible in. The 16
augmented forms born by day 1 are as follows:

\begin{multicols}{3}
    \begin{itemize}
        \item
            $0\coloneq\{\cdot\mid\cdot\}$;
        \item
            $\{\cdot\mid0\}$;
        \item
            $\{\cdot\mid\tomb\}$;
        \item
            $\{\cdot\mid0,\tomb\}$;
        \item
            $1\coloneq\{0\mid\cdot\}$;
        \item
            $*\coloneq\{0\mid0\}$;
        \item
            $\{0\mid\tomb\}$;
        \item
            $\{0\mid0,\tomb\}$;
        \item
            $\{\tomb\mid\cdot\}$;
        \item
            $\{\tomb\mid0\}$;
        \item
            $\{\tomb\mid\tomb\}$;
        \item
            $\{\tomb\mid0,\tomb\}$;
        \item
            $\{\tomb,0\mid\cdot\}$;
        \item
            $\{\tomb,0\mid0\}$;
        \item
            $\{\tomb,0\mid\tomb\}$;
        \item
            $\{\tomb,0\mid0,\tomb\}$.
        \item[]
        \item[]
    \end{itemize}
\end{multicols}

By considering symmetry and reductions (see \cite[Definitions 5.15 \& 5.17 and
Lemma 5.16 on pp.~26--27]{siegel:on}), we need only consider the following 6
forms.

\begin{multicols}{3}
    \begin{itemize}
        \item
            $0$;
        \item
            $1$;
        \item
            $*$;
        \item
            $\{\tomb,0\mid\cdot\}$;
        \item
            $\{\tomb,0\mid0\}$;
        \item
            $\{\tomb,0\mid0,\tomb\}$.
    \end{itemize}
\end{multicols}

It is trivial that 0 is $\mathcal{U}$-invertible for all universes
$\mathcal{U}$ (\cref{prop:0-invertible}). It is useful to note that a game $G$
of formal birthday 1 is $\mathcal{U}$-invertible if and only if
$G+\overline{G}$ is Left $\mathcal{U}$-strong, which is just a corollary of
\cref{thm:invertibility-characterisation}.

\begin{corollary}
    \label{cor:birth-1-invert}
    If $\mathcal{U}$ is a universe and $G\in\hat{\mathcal{U}}$ has formal
    birthday 1, then $G$ is $\mathcal{U}$-invertible if and only if
    $G+\overline{G}$ is Left $\mathcal{U}$-strong.
\end{corollary}

\begin{proof}
    By \cref{thm:invertibility-characterisation}, it suffices to show that
    every option of the $\mathcal{U}$-simplest form of $G$ is
    $\mathcal{U}$-invertible. But $G$ has formal birthday 1, and hence 0 is the
    only possible option of the $\mathcal{U}$-simplest form of $G$, which is
    trivially $\mathcal{U}$-invertible (\cref{prop:0-invertible}).
\end{proof}

It is tempting to think that we should be able to weaken the hypothesis of
\cref{cor:birth-1-invert} to let $\mathcal{U}$ be a monoid that might
\emph{not} be a universe, but this is unclear. Compare
\cref{cor:birth-1-invert} with the following result that we can indeed prove;
in particular, if the monoid has the conjugate property, then we would obtain
essentially the same conclusion.

\begin{proposition}
    \label{prop:birth1}
    If $\mathcal{A}$ is a hereditary set of games and $G\in\maug$ has formal
    birthday 1, then: $G+\overline{G}\equiv_\mathcal{A}0$ if and only if
    $G+\overline{G}$ is Left $\mathcal{A}$-strong.
\end{proposition}

\begin{proof}
    This follows immediately from
    \cref{prop:comp-implies,prop:maint+strong-implies}.
\end{proof}

We now begin to investigate each of the five augmented forms born on day 1 (up
to conjugation and reduction), attempting to determine which universes they are
invertible in.

\subsection{Disintegrators}

We begin with the game $1$. By \cref{cor:birth-1-invert}, we know that $1$ is
$\mathcal{U}$-invertible if and only if $1\in\hat{\mathcal{U}}$ and
$1+\overline{1}$ is Left $\mathcal{U}$-strong. Being Left $\mathcal{U}$-strong
means that Left wins going first on $1+\overline{1}+X$ for all Left ends $X$ in
$\mathcal{U}$. This must be equivalent to Left winning going second on
$\overline{1}+X$. We give a definition to help characterise those $X$ such that
Right will win going first on $\overline{1}+X$.

\begin{definition}
    \label{def:disintegrator}
    If $G\in\mathcal{M}$, then we say that $G$ is a
    \emph{disintegrator}\footnote{The words \emph{disintegrator} and
    \emph{integer} both share the Proto-Indo-European root \emph{*tag-}.} if
    there exists a Right option $G^R$ of $G$ that is not a Left end such that,
    for all Left options $G^{RL}$ of $G^R$, either
    \begin{enumerate}
        \item
            $\outcomeL\left(G^{RL}\right)=\mathscr{R}$; or
        \item
            $G^{RL}$ is a disintegrator.
    \end{enumerate}
\end{definition}

For example, any game with a Right option of the form $\{*\mid G^R\}$ must be a
disintegrator (where $G^R$ can be arbitrary).

The following proposition is the reason disintegrators were defined in this
way, and the reader can likely easily convince themselves of its correctness.
The proof is tedious and just a straightforward application of the definition,
so we have relegated it to \cref{app:disintegrator}.

\begin{proposition}
    \label{prop:disintegrator}
    If $\mathcal{U}$ is a universe and $1\in\hat{\mathcal{U}}$, then 1 is
    $\mathcal{U}$-invertible if and only if $\mathcal{U}$ contains no Left end
    that is a disintegrator.
\end{proposition}

A dream result would be one which says adding two Left ends that are not
disintegrators cannot result in a disintegrator. This is because universes are
characterised by their Left ends (see \cite[Proposition 2.3 on
p.~6]{davies:on}), and we would thus be able to look at a generating set of the
Left ends in a universe and determine whether 1 is invertible. But we must wake
ourselves from this dream, for it is not true. As an example, consider the Left
end $G=\{\cdot\mid\{1\mid0,\overline{1}\}+\{\overline{1},1\mid\cdot\}\}$, where
$G$ is not a disintegrator, but $G+G$ is (which the reader may readily check
via the definition). We will see this example again soon.

\subsection{Starkillers}

We now turn our attention to $*$, which the reader might have guessed from the
title of this subsection. In the same way that we just investigated 1, we give
here a definition to help characterise those Left ends $X$ such that Right wins
going first on $*+X$.

\begin{definition}
    \label{def:starkiller}
    If $G\in\mathcal{M}$, then we say that $G$ is a \textit{starkiller} if
    there exists a Right option $G^R$ of $G$ such that
    $\outcome^R\left(G^R\right)=\mathscr{R}$ and, for all Left options $G^{RL}$
    of $G^R$, either
    \begin{enumerate}
        \item
            $\outcome^L\left(G^{RL}\right)=\mathscr{R}$, or
        \item
            $G^{RL}$ is a starkiller.
    \end{enumerate}
\end{definition}

Much like \cref{prop:disintegrator}, we defined a starkiller precisely to get
the following proposition. Its proof is similarly tedious and straightforward,
so we have relegated it to \cref{app:starkillers}.

\begin{proposition}
    \label{prop:starkiller}
    If $\mathcal{U}$ is a universe, then $*$ is $\mathcal{U}$-invertible if and
    only if $\mathcal{U}$ contains no Left end that is a starkiller.
\end{proposition}

Just as for disintegrators, our dream of having Left ends that aren't
starkillers being closed under addition is quickly shattered. Take, for
example, the Left end
$G=\{\cdot\mid\{1\mid0,\overline{1}\}+\{\overline{1},1\mid\cdot\}\}$, where $G$
is not a starkiller, but $G+G$ is. This is, in fact, the same example we gave
in the previous subsection.

As an example application of \cref{prop:starkiller}, it is a trivial
observation that a terminable Left end must be a starkiller; i.e.\ a Left end
where Right has an option to 0. Since $\overline{1}$ is a terminable Left end,
it is then a simple corollary that, if $\mathcal{U}$ is a dead-ending universe,
then $*$ is $\mathcal{U}$-invertible if and only if $\mathcal{U}=\mathcal{D}$.

\begin{corollary}
    If $\mathcal{U}$ is a dead-ending universe, then, $*$ is
    $\mathcal{U}$-invertible if and only if $\mathcal{U}=\mathcal{D}$. That is,
    the dicot universe is the only dead-ending universe where $*$ is
    invertible.
\end{corollary}

\begin{proof}
    If $\mathcal{U}=\mathcal{D}$, then we know already that $*$ is
    $\mathcal{U}$-invertible. So, assume now that $\mathcal{U}$ is a
    dead-ending universe not equal to $\mathcal{D}$. It is then clear that
    $\overline{1}$ \emph{must} be in $\mathcal{U}$. But $\overline{1}$ is a
    starkiller (we have $\outcomeL(*+*+\overline{1})=\mathscr{R}$), and hence
    $*$ is not $\mathcal{U}$-invertible.
\end{proof}

Immediately from \cref{prop:disintegrator,prop:starkiller}, we have a
characterisation of those universes that contain a $\mathcal{U}$-invertible
element of formal birthday 1, which we state now.

\begin{theorem}
    If $\mathcal{U}$ is a universe, then $\mathcal{U}$ admits a
    $\mathcal{U}$-invertible element (in $\mathcal{U}$) of formal birthday 1 if
    and only if at least one of the following is true:
    \begin{itemize}
        \item
            $\mathcal{U}$ admits no Left ends that are starkillers; or
        \item
            $1\in\mathcal{U}$ and $\mathcal{U}$ admits no Left ends that are
            disintegrators.
    \end{itemize}
\end{theorem}

\subsection{Super starkillers}
\label{subsec:super-starkilers}

We now consider the game $\{\tomb,0\mid0\}$. Writing $G=\{\tomb,0\mid0\}$, it
is useful to note that
\begin{align*}
    G+\overline{G}&\cong\{G,\overline{G}\mid G,\overline{G}\}\\
                  &\equiv_\mathcal{M}\{G\mid\overline{G}\}.
\end{align*}
In the same way as the previous two subsections, we give here a definition to
help characterise those Left ends $X$ such that Right wins going first on
$\{0\mid0,\tomb\}+X$.

\begin{definition}
    \label{def:super-starkiller}
    If $G\in\mathcal{M}$, then we say that $G$ is a \textit{super starkiller}
    if there exists a Right option $G^R$ of $G$ that is not a Left end such
    that $\outcome^R\left(G^R\right)=\mathscr{R}$ and, for all Left options
    $G^{RL}$ of $G^R$, either
    \begin{enumerate}
        \item
            $\outcome^L\left(G^{RL}\right)=\mathscr{R}$, or
        \item
            $G^{RL}$ is a super starkiller.
    \end{enumerate}
\end{definition}

It should be noted that, from a simple comparison of
\cref{def:disintegrator,def:starkiller,def:super-starkiller}, a super
starkiller is necessarily both a distingrator and a starkiller. To give
explicit examples that illustrate the differences:
\begin{center}
    \begin{tabular}{c | c c c}
        & disintegator & starkiller & super starkiller\\
        \hline\\ [-2.5ex]
        $\{\cdot\mid\{*\mid0\}\}$ & \cmark & \xmark & \xmark \\
        $\overline{1}$ & \xmark & \cmark & \xmark\\
        $\{\cdot\mid\{*\mid\overline{1}\}\}$ & \cmark & \cmark & \cmark\\
    \end{tabular}
\end{center}

Just as for disintegrators and starkillers, the proof of the following
proposition is simply a tedious application of the definition, so we relegate
it to \cref{app:super-starkillers}.

\begin{proposition}
    \label{prop:super-starkiller}
    If $\mathcal{U}$ is a universe and $\{\tomb,0\mid0\}\in\hat{\mathcal{U}}$,
    then $\{\tomb,0\mid0\}$ is $\mathcal{U}$-invertible if and only if
    $\mathcal{U}$ contains no Left end that is a super starkiller.
\end{proposition}

To complete the triad of disappointments, consider the Left end
$G=\{\cdot\mid\{1\mid0,\overline{1}\}+\{\overline{1},1\mid\cdot\}\}$, where $G$
is not a super starkiller, but $G+G$ is. Once again, this is the same example
we gave before (a powerful counter-example, indeed, but there exist others,
which the reader is more than welcome to find).

\subsection{The remaining forms}

The augmented forms left to consider are $\{\tomb,0\mid\cdot\}$ and
$\{\tomb,0\mid0,\tomb\}$. Consider first $G=\{\tomb,0\mid\cdot\}$. We calculate
$G+\overline{G}\cong\{\tomb,\overline{G}\mid G,\tomb\}$. Since this is Left
$\mathcal{U}$-strong for all universes $\mathcal{U}$, it is clear from
\cref{prop:birth1} that $G+\overline{G}\equiv_\mathcal{U}0$ for all universes
$\mathcal{U}$. That is, if $G\in\hat{\mathcal{U}}$, then $\mathcal{U}$ admits
an invertible element.

Now consider $G=\{\tomb,0\mid0,\tomb\}$. We calculate
$G+\overline{G}\cong\{\tomb,G\mid G,\tomb\}$. Since this is Left
$\mathcal{U}$-strong for all universes $\mathcal{U}$, we have again by
\cref{prop:birth1} that $G+\overline{G}\equiv_\mathcal{U}0$ for all universes
$\mathcal{U}$. That is, if $G\in\hat{\mathcal{U}}$, then $\mathcal{U}$ admits
an invertible element.

It is here that we must interrupt the narrative to assuage the concerns of our
alarmed readers, for what we have just said may appear strange. We have found
that, for all universes $\mathcal{U}$,
\begin{gather*}
    \{\tomb,0\mid\cdot\}+\{\cdot\mid0,\tomb\}\equiv_\mathcal{U}0,\text{ and}\\
    2\cdot\{\tomb,0\mid0,\tomb\}\equiv_\mathcal{U}0.
\end{gather*}
The informed reader will recall that the full mis\`ere universe $\mathcal{M}$
is reduced, meaning that it has no invertible elements. That is, there exists
no $\mathcal{M}$-invertible $G\in\mathcal{M}$. And therein lies the antidote:
the $\mathcal{M}$-equivalence classes of the games $\{\tomb,0\mid\cdot\}$ and
$\{\tomb,0\mid0,\tomb\}$ don't contain any games in $\mathcal{M}$! They do not
have $\mathcal{M}$-expansions, or, equivalently here, they are not the simplest
forms of any games in $\mathcal{M}$. So, we are \emph{not} saying that these
two games $\{\tomb,0\mid\cdot\}$ and $\{\tomb,0\mid0,\tomb\}$ are
$\mathcal{U}$-invertible for every universe $\mathcal{U}$, but just for those
universes where these games have a $\mathcal{U}$-expansion (equivalently here,
for those universes where these games are the $\mathcal{U}$-simplest forms of
some games in $\mathcal{U}$).

Collecting together
\cref{prop:disintegrator,prop:starkiller,prop:super-starkiller}, we have the
following characterisation of reduced universes. But it is important to note
that this result is not immediately practical. For example, how does one check
whether $\mathcal{U}$ contains a Left end that is a starkiller? How does one
check whether $\{\tomb,0\mid\cdot\}\in\hat{\mathcal{U}}$? These do not appear
to be easy problems to solve.

\begin{theorem}
    \label{thm:reduced}
    If $\mathcal{U}$ is a universe, then it is reduced if and only if the
    following statements are all true:
    \begin{enumerate}
        \item
            if $1\in\hat{\mathcal{U}}$, then there exists a Left end in
            $\mathcal{U}$ that is a disintegrator;
        \item
            there exists a Left end in $\mathcal{U}$ that is a starkiller;
        \item
            if $\{\tomb,0\mid0\}\in\hat{\mathcal{U}}$, then there exists a Left
            end in $\mathcal{U}$ that is a super starkiller;
        \item
            $\{\tomb,0\mid\cdot\}\not\in\hat{\mathcal{U}}$; and
        \item
            $\{\tomb,0\mid0,\tomb\}\not\in\hat{\mathcal{U}}$.
    \end{enumerate}
\end{theorem}

Earlier, in \cref{prob:rank1}, we asked whether there exists a universe
$\mathcal{U}$ that is not reduced such that there exists no
$\mathcal{U}$-invertible $G\in\mathcal{U}$ of formal birthday 1.
\cref{thm:reduced} might present a path to try and tackle such a question.

Recalling that the full mis\`ere universe $\mathcal{M}$ is reduced, it follows
clearly that it must be the unique universe that is maximal with respect to
being reduced. But what about universes that are \emph{minimal} with respect to
being reduced? We will see in the next section (in the discussion after
\cref{cor:universe-weak}) that $\mathcal{D}(\{\cdot\mid2\})$ is one such
universe (since its only subuniverses are $\mathcal{D}$ and
$\mathcal{D}(\overline{1})$, neither of which is reduced). We will also see
that there are uncountably many reduced universes.

\section{Weak universes}
\label{sec:weak}

In any given universe, it is always trivially the case that a Left end-like
form is Left $\mathcal{U}$-strong. But the converse is not always true: in our
favourite universes $\mathcal{D}$ and $\mathcal{E}$, it is trivial to find
games that are Left $\mathcal{D}$- and $\mathcal{E}$-strong (respectively) but
which are not Left ends. For the full mis\`ere universe $\mathcal{M}$, however,
being Left $\mathcal{M}$-strong is exactly equivalent to being Left end-like.
We will call a universe exhibiting this behaviour (like $\mathcal{M}$) a
\emph{weak} universe, since it has no more strong elements than absolutely
required.

\begin{definition}
    \label{def:weak}
    If $\mathcal{A}$ is a set of games such that $G\in\maug$ is Left
    $\mathcal{A}$-strong if and only if $G$ is Left end-like, then we say
    $\mathcal{A}$ is \emph{Left weak}. We define \emph{Right weak} analagously,
    and furthermore say that $\mathcal{A}$ is \emph{weak} if it is both Left
    and Right weak.
\end{definition}

Of course, a set of games that is conjugate closed is Left weak if and only if
it is Right weak; and so it is simply either weak or not.

As we have discussed, the universe $\mathcal{M}$ is weak---and we will give
more examples soon. But we first give a short proof for completeness and to
highlight this fact.

\begin{proposition}
    The universe $\mathcal{M}$ is weak.
\end{proposition}

\begin{proof}
    Let $G\in\maug$ be an augmented form that is not Left end-like. Let
    $X=\{\cdot\mid\birth(G)\}$, and consider Left playing first on $G+X$. Since
    $G$ is not Left end-like, Left does not win immediately and so she must
    play to some $G^L+X$. Right can respond to $X+\birth(G)$, and by
    construction will clearly run out of moves before Left, hence winning the
    game. Thus, $G$ is not Left $\mathcal{M}$-strong.
\end{proof}

It is not shocking that all weak universes induce the same relation; in
particular, all weak universes induce the same partial order as $\mathcal{M}$
does (with $\geq_\mathcal{M}$).

\begin{proposition}
    \label{prop:weak-agree}
    If $\mathcal{U}$ and $\mathcal{W}$ are weak universes, then the relations
    $\geq_\mathcal{U}$ and $\geq_\mathcal{W}$ agree on $\maug$.
\end{proposition}

\begin{proof}
    From \cref{thm:comparison}, we observe that two universes induce the same
    relation on $\maug$ if they have the same proviso. Since a game in $\maug$
    is Left $\mathcal{U}$-strong if and only if it is Left
    $\mathcal{W}$-strong, it is clear that $\geq_\mathcal{U}$ and
    $\geq_\mathcal{W}$ have the same proviso, and hence they agree on $\maug$;
    i.e.\ they are the same relation.
\end{proof}

It is also easy to see that every weak monoid of games is also a reduced
monoid. More precisely: every invertible element of a Left (Right) weak monoid
is a Left (Right) end.

\begin{proposition}
    \label{prop:weak-implies-red}
    If $\mathcal{A}$ is a Left (Right) weak monoid of games, then every element
    of the invertible subgroup of $\mathcal{A}$ is a Left (Right) end. In
    particular, if $\mathcal{A}$ is weak, then $\mathcal{A}$ is reduced.
\end{proposition}

\begin{proof}
    Let $G,H\in\mathcal{A}$ with $G+H\equiv_\mathcal{A}0$ and $\mathcal{A}$
    Left weak. By \cref{prop:comp-implies}, it follows that $G+H$ is Left
    $\mathcal{A}$-strong. Since $\mathcal{A}$ is Left weak, we must have that
    $G+H$ is Left end-like. But $G$ and $H$ have no tombstone options, so $G+H$
    is a Left end, which implies $G$ and $H$ are both Left ends. The result for
    $\mathcal{A}$ being Right weak follows by symmetry. Finally, if
    $\mathcal{A}$ is weak, then our $G$ and $H$ from above would necessarily be
    both Left ends and Right ends, which implies they are both isomorphic to 0,
    yielding the result.
\end{proof}

In particular, a weak universe is reduced, which we simply state now. We could
also have realised this from \cref{prop:weak-agree}; every weak universe
induces the same partial order as $\mathcal{M}$, which is reduced, and hence
every weak universe is reduced.

\begin{corollary}
    \label{cor:weak-implies-red}
    If $\mathcal{U}$ is a weak universe, then $\mathcal{U}$ is reduced.
\end{corollary}

What is unclear, however, is whether the reverse implication holds. That is, if
we have a reduced universe, must it also be weak? We have found no such
examples, and leave it as an open problem.

\begin{problem}
    \label{prob:reduced-weak}
    Must a reduced universe be weak?
\end{problem}

The reader may or may not be surprised to learn that there exist games whose
mere presence in a a set of games forces the set to be Left weak (or Right
weak). We call such games \emph{Left weakening} (respectively \emph{Right
weakening}).

\begin{definition}
    If $X\in\mathcal{M}$ is a Left end (Right end) such that every set of
    games containing $X$ is Left weak (Right weak), then we say $X$ is
    \emph{Left weakening} (\emph{Right weakening}). If is both Left and Right
    weakening, then we simply say it is \emph{weakening}.
\end{definition}

Clearly a set of games that is conjugate closed contains a Left weakening end
if and only if it contains a Right weakening end, and so (in the context of
such a set of games) we need only discuss whether a Left end is weakening or
not.

What is not clear, however, is whether a weak universe necessarily contains a
weakening Left end. (Or, more generally, whether a Left weak monoid necessarily
contains a Left weakening Left end.) We leave this as an open problem.

\begin{problem}
    \label{prob:weak-weakening}
    Is it true that a universe is weak if and only if it contains a weakening
    Left end, or is it possible to have a weak universe that does not contain
    any weakening Left ends?
\end{problem}

For monoids of games, a simple example of a Left weakening end is the form
$\{\cdot\mid2\}$. A key property of this game is that Right can effectively
\textit{give back} more moves to Left than Left can give back to him; in
particular here, Right has to play only 1 move in order to give back 2 to Left.

\begin{theorem}
    \label{thm:weakening-game}
    If $\mathcal{A}$ is a monoid of games containing $\{\cdot\mid2\}$, then
    $\mathcal{A}$ is Left weak.
\end{theorem}

\begin{proof}
    Let $G$ be an augmented form that is not Left end-like. Now consider Left
    playing first on $G+\birth(G)\cdot\{\cdot\mid2\}$. We will show that Right
    wins.

    Left must play to some $G^L+\birth(G)\cdot\{\cdot\mid2\}$. Right can then
    respond with $G^L+2+(\birth(G)-1)\cdot\{\cdot\mid2\}$. Regardless of what
    Left does, Right will continue to play on the $\{\cdot\mid2\}$ components
    until they are exhausted. Clearly Left cannot run out of moves before Right
    accomplishes this task. At the moment that Right plays on the last
    $\{\cdot\mid2\}$ component, we have a game of the form $G'+n$, where $G'$
    is a subposition of $G^L$, and $n\geq\birth(G)+1$. Thus, Right will win.
\end{proof}

An almost identical proof clearly yields a countably infinite number of Left
weakening ends. For example, \emph{every} Left end with an option to 2 must be
Left weakening. And it needn't have 2 as an option either; it could instead
have an option to some $n>2$. But even this is not exhaustive, and the reader
may amuse themselves finding other more complex examples. We leave it as an
open problem to try and characterise such games.

\begin{problem}
    \label{prob:characterise-weakening}
    Find a characterisation of the weakening Left ends.
\end{problem}

For universes in particular, we have the following obvious corollary of
\cref{thm:weakening-game}, which states that a universe containing
$\{\cdot\mid2\}$ must be weak (and reduced).

\begin{corollary}
    \label{cor:universe-weak}
    If $\mathcal{U}$ is a universe containing $\{\cdot\mid2\}$, then
    $\mathcal{U}$ is weak (and hence also reduced).
\end{corollary}

\begin{proof}
    It follows from \cref{thm:weakening-game} that $\mathcal{U}$ is Left weak.
    Now, every universe is conjugate closed, and hence $\mathcal{U}$ must also
    be Right weak. That is, $\mathcal{U}$ is weak, and hence also reduced
    (\cref{cor:weak-implies-red}).
\end{proof}

It then follows swiftly that $\mathcal{D}(\{\cdot\mid2\})$ must be a universe
that is minimal with respect to being reduced, since its only subuniverses are
$\mathcal{D}$ and $\mathcal{D}(\overline{1})$, neither of which is reduced. We
can also give a very simple construction for uncountably many weak universes
here (and hence also uncountably many reduced universes). Take any subset
$\mathcal{S}$ of $\mathbb{N}_{\geq3}$, and then construct the universal closure
of the Left ends $\{\{\cdot\mid2,n\}:n\in\mathcal{S}\}$. This defines an
injective map from the subsets of $\mathbb{N}_{\geq3}$ to the set of universes,
and hence there must be uncountably many weak universes.

Going to back to monoids of games (that aren't necessarily universes), we have
another corollary of \cref{thm:weakening-game}.
\begin{corollary}
    \label{cor:left-end-invert}
    If $\mathcal{A}$ is a monoid of games containing $\{\cdot\mid2\}$, then
    every $\mathcal{A}$-invertible game is a Left end. In particular,
    $\mathcal{A}^{\overline{\times}}$ is the trivial group.
\end{corollary}

\begin{proof}
    This follows immediately from \cref{prop:weak-implies-red}.
\end{proof}

Note that \cref{cor:left-end-invert} does not mean that the invertible subgroup
of $\mathcal{A}$ is trivial---although, the conjugate invertible subgroup of
$\mathcal{A}$ certainly is. Indeed, in Milley's monoid of \textsc{partizan
kayles} $\mathcal{K}$, it is the case that
$\overline{1}+W_2\equiv_\mathcal{K}0$ \cite[Corollary 4 on
p.~11]{milley:partizan}, and it is obvious that
$\mathcal{K}^{\overline{\times}}$ is the trivial group. And notice that both
$\overline{1}$ and $W_2$ are Left (dead) ends. But, if we want to give an
example of a Left weak monoid that has a non-trivial invertible subgroup, then
$\mathcal{K}$ will not do it for us, since it is not Left weak: recall that
$\mathcal{K}\subseteq\mathcal{E}$, and also that
$1+\overline{1}\equiv_\mathcal{E}0$; in particular, $1+\overline{1}$ is Left
$\mathcal{E}$-strong, and hence also Left $\mathcal{K}$-strong; but
$1+\overline{1}$ is not Left end-like, so $\mathcal{K}$ cannot be Left weak. We
will now give a true example, demonstrating further how strange mis\`ere
monoids can be.

Let $\mathcal{J}=\left\langle \overline{1}, \{\cdot\mid2\}\right\rangle$; that
is, the semigroup generated by the (additive) closure of $\overline{1}$ and
$\{\cdot\mid2\}$ (we will show soon that it is a monoid). Observe that
\[
    \outcome(\overline{n}+m\cdot\{\cdot\mid2\})=\begin{cases}
        \mathscr{L} & n>m\\
        \mathscr{N} & n\leq m.
    \end{cases}
\]
It then clearly follows that
$\overline{n}+m\cdot\{\cdot\mid2\}\equiv_\mathcal{J}0$ if and only if $n=m$.
Hence $\mathcal{J}$ must be a monoid: $\overline{1}+\{\cdot\mid2\}$ is the
identity. Since $\{\cdot\mid2\}\in\mathcal{J}$, we know that $\mathcal{J}$ is
Left weak by \cref{thm:weakening-game}. Now we have
\begin{align*}
    \mathcal{J}&\cong\left\langle x,y\mid xy\right\rangle\\
               &\cong\mathbb{Z}.
\end{align*}
So, our Left weak monoid $\mathcal{J}$ is not only \emph{not} reduced, but it
is a group! We will stress again that such a thing cannot occur in a universe
(or indeed any mis\`ere monoid with the conjugate property); if a universe is
weak, then it is reduced.

Regarding Milley's monoid of \textsc{partizan kayles} $\mathcal{K}$, they write
in \cite[p.~13]{milley:partizan} that it is the only known partizan example
where a game may have an inverse that is not its conjugate. We have
demonstrated, with our strange monoid $\mathcal{J}$, another example. Indeed,
using the ideas we have discussed, it is not difficult to find many more. But
we make a remark on the invertibility here. Milley writes
$\overline{1}+W_2\equiv_\mathcal{K}0$, and observes that
$1\not\equiv_\mathcal{K}W_2$. Now, even though $\overline{1}$ and $W_2$ are not
conjugate $\mathcal{K}$-invertible, they do satisfy the following relations:
\begin{align*}
    1+\overline{1}&\equiv_\mathcal{K}0,\\
    W_2+\overline{W_2}&\equiv_\mathcal{K}0.
\end{align*}
But, in our monoid $\mathcal{J}$, we have a stronger behaviour:
\begin{align*}
    1+\overline{1}&\not\equiv_\mathcal{J}0,\\
    \{\cdot\mid2\}+\overline{\{\cdot\mid2\}}&\not\equiv_\mathcal{J}0.
\end{align*}
This is because $\mathcal{J}$ is Left weak, and hence only Left ends might be
equal to 0 ($\mathcal{K}$ is \emph{not} Left weak). It would be useful to
collate a number of examples and their behaviours so that we may get a better
lay of the land with regards to the strangeness of partizan mis\`ere monoids.

\section{Final remarks}
\label{sec:final}

In his brilliant paper, Siegel laments the possible lack of utility of his
simplest forms \cite[\S7 on pp.~34--35]{siegel:on}. Indeed, in the subsection
``Is any of this useful for anything?'', he gives some computational evidence
for the small amounts of information we are able to discard from a game tree,
and he writes that the theory of simplest forms is ``unlikely to be applicable
to specific case studies in a way that provides much real insight.'' We have
seen in this paper that the simplest forms were instrumental in coordinating
our proofs. We now know far more about invertibility in universes, which will
indeed have practical benefit to analysing specific rulesets. As such, we push
back: both in theory and in practice, Siegel's simplest forms are wonderful.

It is curious how the development of the simplest forms led so swiftly to our
results on the conjugate property and the characterisation of the invertible
elements. It would be interesting to construct an alternative proof that makes
no use of these forms. This has, of course, been tried before, but the
challenge of overcoming end-reversibilty is a great one. It is left as a
challenge to the reader.

\begin{problem}
    Can we find a proof of each universe having the conjugate property without
    using simplest forms? And the characterisation of the invertible elements?
\end{problem}

Given that we do now have a characterisation of the invertible elements of each
universe, an almost entirely open area that is ripe for future work is
understanding what their group structure is (the invertible elements of a
monoid always form a group). These groups must be countable and abelian, of
course, but what else can we say? It might be wise to look at the subgroups
generated by those invertible games born by days 2 and 3, and see what patterns
emerge.

\begin{problem}
    What can we say about the group structure of the invertible subgroup of
    each universe?
\end{problem}

We have scattered a number of other open problems throughout this paper. For
the reader's convenience, we list them here:
\begin{center}
    \begin{tabular}{c}
        \cref{prob:inverses} on \cpageref{prob:inverses}\\
        \cref{prob:rank1} on \cpageref{prob:rank1}\\
        \cref{prob:reduced-weak} on \cpageref{prob:reduced-weak}\\
        \cref{prob:weak-weakening} on \cpageref{prob:weak-weakening}\\
        \cref{prob:characterise-weakening} on
        \cpageref{prob:characterise-weakening}
    \end{tabular}
\end{center}
Many seem tantalising. The likeliest to bear fruit quickly is
\cref{prob:characterise-weakening}: characterising the weakening Left ends. We
hope to see many advancements soon!

\bibliographystyle{plainurl}
\bibliography{bib}

\appendix

\section{}

\subsection{Disintegrators}
\label{app:disintegrator}

\begin{lemma}
    \label{lem:disintegrator}
    If $X$ is a disintegrator, then $\outcomeR(\overline{1}+X)=\mathscr{R}$.
\end{lemma}

\begin{proof}
    Since $X$ is a disintegrator, there exists some $X^R$ that is not a Left
    end such that, for all Left options $X^{RL}$, either
    \begin{enumerate}
        \item $\outcomeL(X^{RL})=\mathscr{R}$; or
        \item $X^{RL}$ is a disintegrator.
    \end{enumerate}
    Right should play to this $X^R$, leaving $\overline{1}+X^R$. Left must then
    respond to some $\overline{1}+X^{RL}$ (recall that $X^R$ is not a Left
    end). If $\outcome(X^{RL})=\mathscr{R}$, then Right can play on
    $\overline{1}$ to 0, leaving himself with a winning position. Instead, if
    $X^{RL}$ is a disintegrator, then Right wins by induction.
\end{proof}

\begin{lemma}
    \label{lem:not-disintegrator}
    If $X$ satisfies $\outcomeL(X)=\mathscr{L}$ and is not a disintegrator,
    then $\outcomeR(\overline{1}+X)=\mathscr{L}$.
\end{lemma}

\begin{proof}
    Right cannot play on $\overline{1}$ to 0, leaving $X$, since this is
    winning for Left. So, Right must play on $X$ to some $X^R$, leaving
    $\overline{1}+X^R$. Since $X$ is not a disintegrator, we must be in one of
    two cases:
    \begin{enumerate}
        \item $X^R$ is a Left end.

            In this case, Left wins immediately since she has no moves.
        \item There exists a Left option $X^{RL}$ such that
            $\outcomeL(X^{RL})=\mathscr{L}$ and $X^{RL}$ is not a
            disintegrator.

            In this case, Left can play to this $X^{RL}$, which must be winning
            by induction.
    \end{enumerate}
\end{proof}

\begin{namedthm*}{\cref{prop:disintegrator}}
    If $\mathcal{U}$ is a universe and $1\in\hat{\mathcal{U}}$, then 1 is
    $\mathcal{U}$-invertible if and only if $\mathcal{U}$ contains no Left end
    that is a disintegrator.
\end{namedthm*}

\begin{proof}
    By \cref{cor:birth-1-invert}, we know that $1$ is $\mathcal{U}$-invertible
    if and only if $1+\overline{1}$ is Left $\mathcal{U}$-strong, which means
    that $\outcomeL(1+\overline{1}+X)=\mathscr{L}$ for all Left ends
    $X\in\mathcal{U}$. Since $X$ is a Left end, this is equivalent to saying
    that $\outcomeR(\overline{1}+X)=\mathscr{L}$ for all Left ends
    $X\in\mathcal{U}$.

    The result then follows immediately from
    \cref{lem:disintegrator,lem:not-disintegrator}.
\end{proof}

\subsection{Starkillers}
\label{app:starkillers}

\begin{lemma}
    \label{lem:starkiller}
    If $X$ is a starkiller, then $\outcomeR(*+X)=\mathscr{R}$.
\end{lemma}

\begin{proof}
    By \cref{def:starkiller}, we know that there exists some Right option $X^R$
    such that $\outcomeR(X^R)=\mathscr{R}$ and, for all Left options $X^{RL}$,
    either
    \begin{enumerate}
        \item
            $\outcomeL(X^{RL})=\mathscr{R}$; or
        \item
            $X^{RL}$ is a starkiller.
    \end{enumerate}
    So, let us consider Right moving to $*+X^R$. If Left plays on $*$ to leave
    simply $X^R$, then Right wins since $\outcome(X^R)=\mathscr{R}$. Otherwise,
    Left moves to some $*+X^{RL}$. If $\outcomeL(X^{RL})=\mathscr{R}$, then
    Right can play on $*$ to leave $X^{RL}$, which is thus winning. Otherwise,
    $X^{RL}$ must be a starkiller, and so Right wins by induction.
\end{proof}

\begin{lemma}
    \label{lem:not-starkiller}
    If $X$ satisfies $\outcomeL(X)=\mathscr{L}$ and is not a starkiller, then
    $\outcomeR(*+X)=\mathscr{L}$.
\end{lemma}

\begin{proof}
    Right cannot play on $*$ to 0, leaving $X$, otherwise Left would win
    immeidately. Thus, Right must play on $X$ to leave some $*+X^R$. Since $X$
    is \emph{not} a starkiller, it follows (straight from
    \cref{def:starkiller}) that either $\outcomeR(X^R)=\mathscr{L}$, or else
    there exists some $X^{RL}$ such that $\outcomeL(X^{RL})=\mathscr{L}$ and
    $X^{RL}$ is not a starkiller.

    If we are in the first case, with $\outcomeR(X^R)=\mathscr{L}$, then Left
    can play on $*$ to leave $X^R$, which must be winning. In the second case,
    Left should move on $X^R$ to the $X^{RL}$ with properties just discussed,
    leaving $*+X^{RL}$, which must be winning by induction.
\end{proof}

\begin{namedthm*}{\cref{prop:starkiller}}
    If $\mathcal{U}$ is a universe, then $*$ is $\mathcal{U}$-invertible if and
    only if $\mathcal{U}$ contains no Left end that is a starkiller.
\end{namedthm*}

\begin{proof}
    By \cref{cor:birth-1-invert}, we know that $*$ is $\mathcal{U}$-invertible
    if and only if $*+*$ is Left $\mathcal{U}$-strong, which means that
    $\outcomeL(*+*+X)=\mathscr{L}$ for all Left ends $X\in\mathcal{U}$. Since
    $X$ is a Left end, this is equivalent to saying that
    $\outcomeR(*+X)=\mathscr{L}$ for all Left ends $X\in\mathcal{U}$.

    The result then follows immediately from
    \cref{lem:starkiller,lem:not-starkiller}.
\end{proof}

\subsection{Super starkillers}
\label{app:super-starkillers}

\begin{lemma}
    \label{lem:super-starkiller}
    If $X$ is a super starkiller, then
    $\outcomeR(\{\tomb,0\mid0\}+X)=\mathscr{R}$.
\end{lemma}

\begin{proof}
    Since $X$ is a super starkiller, Right can play on $X$ to some $X^R$ that
    is not a Left end such that $\outcomeR(X^R)=\mathscr{R}$ and, for all Left
    options $X^{RL}$ of $X^R$,
    either
    \begin{enumerate}
        \item
            $\outcomeL(X^{RL})=\mathscr{R}$, or
        \item
            $X^{RL}$ is a super starkiller.
    \end{enumerate}
    Since $X$ is not a Left end, the tombstone of $\{\tomb,0\mid0\}$ does not
    yield that Left wins going first here. If Left plays on $\{\tomb,0\mid0\}$
    to 0, leaving $X^R$, then Right wins since $\outcomeR(X^R)=\mathscr{R}$.
    So, Left must instead play on $X^R$ to some $X^{RL}$. If
    $\outcomeL(X^{RL})=\mathscr{R}$, then Right can play on $\{\tomb,0\mid0\}$
    to 0 and win the game. Otherwise, if $X^{RL}$ is a super starkiller, then
    Right wins by induction.
\end{proof}

\begin{lemma}
    \label{lem:not-super-starkiller}
    If $X$ satisfies $\outcomeL(X)=\mathscr{L}$ is not a super starkiller, then
    $\outcomeR(\{\tomb,0\mid0\}+X)=\mathscr{L}$.
\end{lemma}

\begin{proof}
    Right cannot play on $\{\tomb,0\mid0\}$ to 0, otherwise Left would win. So,
    Right must play to some $\{\tomb,0\mid0\}+X^R$. Since $X$ is not a super
    starkiller, we must be in one of the following cases:
    \begin{enumerate}
        \item $X^R$ is a Left end.

            In this case, since $\{\tomb,0\mid0\}$ is Left $\mathcal{U}$-strong
            (it has a Left tombstone), we observe that Left wins going first.
        \item $\outcomeR(X^R)=\mathscr{L}$.

            In this case, Left can play on $\{\tomb,0\mid0\}$ to 0, leaving
            herself a winning position.
        \item there exists a Left option $X^{RL}$ such that
            $\outcomeL(X^{RL})=\mathscr{L}$ and $X^{RL}$ is not a super
            starkiller.

            In this case, Left can play to this $X^{RL}$, which must be winning
            by induction.
    \end{enumerate}
\end{proof}

\begin{namedthm*}{\cref{prop:super-starkiller}}
    If $\mathcal{U}$ is a universe and $\{\tomb,0\mid0\}\in\hat{\mathcal{U}}$,
    then $\{\tomb,0\mid0\}$ is $\mathcal{U}$-invertible if and only if
    $\mathcal{U}$ contains no Left end that is a super starkiller.
\end{namedthm*}

\begin{proof}
    By \cref{cor:birth-1-invert}, we know that $\{\tomb,0\mid0\}$ is
    $\mathcal{U}$-invertible if and only if $\{\tomb,0\mid0\}+\{0\mid0,\tomb\}$
    is Left $\mathcal{U}$-strong. We calculate (or recall from
    \cref{subsec:super-starkilers}) that
    \[
        \{\tomb,0\mid0\}+\{0\mid0,\tomb\}\equiv_\mathcal{M}\{\{\tomb,0\mid0\}\mid\{0\mid0,\tomb\}\}.
    \]
    So, $\{\tomb,0\mid0\}$ is $\mathcal{U}$-invertible if and only if
    $\outcomeR(\{\tomb,0\mid0\}+X)=\mathscr{L}$ for all Left ends
    $X\in\mathcal{U}$.

    The result then follows immediately from
    \cref{lem:super-starkiller,lem:not-super-starkiller}.
\end{proof}

\end{document}